\documentclass[onefignum,onetabnum]{siamart220329}

\usepackage{amsfonts}
\usepackage{graphicx}
\usepackage{epstopdf}
\usepackage{algorithmic}
\usepackage{todonotes}
\usepackage{mathabx}
\usepackage{url, graphicx, xcolor, hyperref, geometry}
\usepackage{latexsym,amsbsy}
\usepackage{lipsum}
\usepackage{amsfonts}
\usepackage{amssymb,amsmath}
\usepackage{smartdiagram}
\usepackage{mathabx}
\usepackage{float} 
\usepackage{enumerate}
\usepackage{bm}
\usepackage{slashbox}
\usepackage{booktabs,multirow}
\newcommand{\cmmnt}[1]{}

\hypersetup{
	colorlinks,
	linkcolor={blue},
	citecolor={blue},
	urlcolor={blue}
}

\newcommand{\yh}[1]{{ \color{black}   #1}}
\newcommand{\rev}[1]{{ \color{black} #1}}
\newcommand\bu{\bm{u}} 
\newcommand\bv{\bm{v}} 
\newcommand\bV{\bm{V}} 

\ifpdf
  \DeclareGraphicsExtensions{.eps,.pdf,.png,.jpg}
\else
  \DeclareGraphicsExtensions{.eps}
\fi

\newsiamremark{remark}{Remark}
\newsiamremark{hypothesis}{Hypothesis}
\crefname{hypothesis}{Hypothesis}{Hypotheses}
\newsiamthm{claim}{Claim}

\headers{Preconditioner for the  grad-div stabilized equal-order FEM of Oseen Problem}{Y. He and M. Olshanskii}

\title{A preconditioner for the  grad-div stabilized equal-order finite elements discretizations of the Oseen problem\thanks{Submitted to the editors DATE. \funding{The author M.O. was partially supported by National Science Foundation under Grants No. DMS-2309197 and  DMS-2408978}}}

\author{Yunhui He\thanks{Department of Mathematics, University of Houston, 3551 Cullen Blvd, Houston, Texas 77204-3008, USA,  (\email{yhe43@central.uh.edu}, \email{maolshanskiy@uh.edu})},  \and  Maxim Olshanskii\footnotemark[2]}

\usepackage{amsopn}

\ifpdf
\hypersetup{
  pdftitle={A preconditioner for the  grad-div stabilized equal-order finite elements discretizations of the Oseen problem},
  pdfauthor={Y. He, and M. Olshanskii}
}
\fi

\begin{document}

\maketitle

\begin{abstract}
The paper considers grad-div stabilized equal-order finite elements (FE) methods for the linearized Navier-Stokes equations.  
A block triangular preconditioner for the resulting system of algebraic equations is proposed which is closely related to the Augmented Lagrangian (AL) preconditioner. A field-of-values analysis of a preconditioned Krylov subspace method shows convergence bounds that are independent of the mesh parameter variation. Numerical studies support the theory and demonstrate  the robustness of the approach also with respect to the viscosity parameter variation, as is typical for AL preconditioners when applied to inf-sup stable FE pairs. The numerical experiments also address the accuracy of grad-div stabilized equal-order FE method for the steady state Navier-Stokes equations.  

\end{abstract}



	\section{Introduction}
	
	The numerical solution of a system of linear algebraic equations resulting from discretizing the Oseen problem is one of the central problems in numerical linear algebra for computational fluid dynamics. The matrix of the system is large, sparse, indefinite, and can be poorly conditioned with a dominating skew-symmetric part at higher Reynolds numbers. Developing efficient algebraic solvers, which means scalable and robust with respect to physical parameters, for this system is a formidable challenge. This challenge has been addressed by many authors from different angles since the mid-90s.
	
	The notable achievements are the development of block preconditioners~\cite{elman2014finite} complemented with Schur complement preconditioners such as pressure convection-diffusion (PCD) and least-square commutator (aka BFBt)\cite{kay2002preconditioner,elman1999preconditioning}, and the introduction of the Augmented Lagrangian (AL) approach \cite{benzi2006augmented}.
	
	Among known approaches, only the Augmented Lagrangian method is recognized to deliver fast convergence rates that are essentially independent of the Reynolds number, the critical physical parameter. This robustness property, documented in multiple studies (e.g.,\cite{benzi2006augmented,de2007two,ur2008comparison,he2011augmented,borm2012,farrell2019augmented,farrell2021reynolds,olshanskii2022recycling,lohmann2024augmented}), is not unconditional. It is known only for inf-sup stable finite element discretizations and relies on the availability of an efficient preconditioner for the velocity block of the augmented system. This condition restricts the effective use of the AL approach to either moderately sized systems, where an (incomplete) factorization of the velocity block is feasible~\cite{olshanskii2022recycling}, or to certain finite element pairs and grids that are amenable to specialized multigrid techniques~\cite{schoberl1999multigrid,benzi2006augmented,farrell2019augmented,farrell2021reynolds}. 
	
	For stabilized equal-order finite element discretizations of the Oseen problem, the Augmented Lagrangian approach is much less studied. As far as we are aware,~\cite{benzi2011modified} is the only paper that introduces an extension for equal-order discretizations. While effective, this extension imposes restrictions on the augmentation parameter, making it not completely robust with respect to the Reynolds number and mesh size variation.
	
	The present paper focuses on building an efficient algebraic solver for the Oseen problem discretized with the equal-order finite element method. Our approach is closely related to the AL method but differs from that in \cite{benzi2011modified}. In that paper, the augmentation is done at the algebraic level, leading to the pressure stabilization matrix appearing in the new velocity operator. Here, the augmentation is done at the continuous level, rendering the method as grad-div stabilization. This modification addresses two issues: it removes algebraic constraints on the augmentation parameter and significantly reduces the fill-in of the velocity matrix compared to the algebraic augmentation. The close relationship between the AL method and grad-div stabilization is well-known, and improving algebraic solvers by adding grad-div stabilization to incompressible fluid problems is a well-studied topic (see, e.g.,\cite{glowinski1989augmented,kobel1995solving,olshanskii2004grad,olshanskii2002low,heister2013efficient,le2013preconditioning,moulin2019augmented}).

	Existing studies of algebraic solvers based on grad-div stabilization have dealt with either a `continuous' formulation (no discretization involved) or fluid problems discretized using inf-sup stable elements such as Taylor–Hood or Scott–Vogelius elements. This might be surprising since the first paper where the grad-div term appears is \cite{franca1992stabilized}, on equal-order finite elements for the incompressible Navier–Stokes equations. The grad-div stabilization acts as a compensation mechanism for under-resolved pressure variables, as explained in \cite{olshanskii2002low,olshanskii2009grad}. Since under-resolved pressure is more typical for inf-sup stable elements, most subsequent studies of grad-div stabilization have focused on inf-sup stable discretizations. The present paper aims to rectify this by introducing an AL-type preconditioner for stabilized equal-order discretizations of the Oseen problem and assessing the accuracy of grad-div stabilized equal-order finite elements for the incompressible Navier–Stokes equations.
	
	Although this paper addresses the accuracy of grad-div stabilized equal-order finite elements in its numerical experiments section, the main focus is on the algebraic solver. Thus, the primary theoretical contribution is a convergence analysis of the GMRES method with the proposed preconditioners by deriving field-of-values estimates for the preconditioned matrices. The analysis of iterative methods for linearized fluid problems often takes the form of eigenvalue bounds for the preconditioned matrices. While useful, these bounds do not provide much insight into the convergence behavior of preconditioned Krylov subspace methods for non-normal matrices. Rigorous analysis of preconditioned Krylov subspace methods for the discrete Oseen problem is scarce, with convergence estimates found in \cite{klawonn1999block,loghin2004analysis,benzi2011field}. Only \cite{benzi2011field} addresses the convergence analysis of GMRES with the augmented Lagrangian preconditioner. Our analysis differs from that in \cite{benzi2011field} in both the problem in question and the arguments used to derive the field-of-values estimates.
	
	In summary, the principal contributions of this paper are: (i) Introducing, for the first time, an AL-type preconditioner based on continuous-level augmentation for the discrete Oseen problem resulting from pressure-stabilized finite element methods; (ii) Proving field-of-values estimates for the preconditioned system, leading to optimal convergence bounds for the GMRES algorithm applied to solve the system; (iii) Conducting numerical experiments to assess both the performance of the algebraic solver and the accuracy of the grad-div stabilized equal-order finite element method for the steady-state incompressible Navier–Stokes problem.

	The remainder of the paper is organized as follows. Section \ref{sec:discretization} introduces the Oseen problem and a finite element formulation with some auxiliary results. In section \ref{sec:algebraic-preconditioners}, we consider an algebraic system and a pressure Schur complement preconditioner. We also derive some useful eigenvalue estimates for the preconditioned Schur complement matrix. Section \ref{sec:field-of-values2} provides a field-of-values analysis for a preconditioner for the full system. Numerical results are presented to illustrate our theoretical findings in section \ref{sec:num}. Concluding remarks are collected in section \ref{sec:con}.
	
	\section{Finite element formulation}\label{sec:discretization}
	
	Let $\Omega \subset \mathbb{R}^d$ with $d=2,3$ be a bounded  polygonal or polyhedral domain.
	In this work, we are interested in numerical solutions to  the Oseen problem: Given a smooth divergence free vector field $\bm{a}:\Omega:\to\mathbb{R}^d$ and a force field $\bm{f}:\Omega:\to\mathbb{R}^d$, find the velocity field  $\bm{u}:\Omega:\to\mathbb{R}^d$ and the scalar  pressure
	function  $p:\Omega:\to\mathbb{R}$ solving the following system: 
	\begin{align*}
		-\nu \Delta \bm{u} +(\bm{a}\cdot \nabla) \bm{u} +\nabla p &=\bm{f} \quad \text{in}\quad \Omega, \\
		{\rm div}\bm{u} &=0~ \quad \text{in}\quad \Omega, \\
		\bm{u}&=\bm{0}\, \quad \text{on}\quad \partial\Omega.
	\end{align*} 
	We assume homogeneous Dirichlet boundary conditions for convenience, but the method and analysis extend to other common boundary conditions. The pressure is defined up to an additive constant.

	To formulate variational and finite element problems, we need the following bilinear forms: 
	\begin{align*}
		a(\bm{u},\bm{v})=\nu(\nabla \bm{u}, \nabla \bm{v}) + (\bm{a} \cdot \nabla \bm{u}, \bm{v}),\quad b(\bm{v},q)= -(q,{\rm div}\bm{v}),\quad\text{with}~\bu, \bv\in H^1_0(\Omega)^d,~q\in L^2(\Omega).
	\end{align*}
	Here and further, $(\cdot,\cdot)$ and $\|\cdot\|$ denote the $L^2(\Omega)$ inner product and norm, respectively.  \yh{Note that $a(\bm{u},\bm{u})= \nu(\nabla \bm{u}, \nabla \bm{u})$, since the convection term is skew-symmetric.}
	
	The weak formulation of the Oseen problem then reads: Find $\bu\in  (H_{0}^1(\Omega))^d$, $p\in L^2(\Omega)$ such that 
	\begin{align*}
		a(\bm{u},\bm{v}) +b(\bm{v},p) &=\langle\bm{f},\bm{v}\rangle, \\
		b(\bm{u},q)& =0,
	\end{align*}
	for any $\bv\in  (H_{0}^1(\Omega))^d$, $q\in L^2(\Omega)$.
	
	We assume a family $ \{\mathcal{T}_h\}_{h>0} $ of quasi-uniform regular tessellations 
	of $\Omega$, parameterized with the discretization parameter $h>0$.  
	For any integer $k$, we define
	\begin{equation*}
		R_k(T) =\begin{cases}
			P_k(T) &  \text{if $T$ is a triangle or tetrahedron},\\
			Q_k(T)  &  \text{if $T$ is a quadrilateral or hexahedron},
		\end{cases}
	\end{equation*}
	where $P_k(T)$ is the space of polynomials of total degree at most $k$ on an element $T$, and $Q_k(T)$ is the space of tensor product polynomials of degree
	at most $k$ on an element $T$.

	In this paper, we are interested in  equal-order finite element approximations of the Oseen problem.   Therefore, velocity and pressure finite element spaces $\bm{V}_h$ and $Q_h$ are defined as
	\begin{align*}
		\bm{V}_h  &= \{ \bm{v}_h \in (H_{0}^1(\Omega))^d :  \bm{v}_h |_T \in R_k(T)^d\,\, \forall T\in \mathcal{T}_h \}, \\
		Q_h  &= \{ q_h \in L^2_0(\Omega): q_h |_T \in R_k(T)\,\, \forall T\in \mathcal{T}_h \}.
	\end{align*}
	Equal order finite elements require a proper stabilization~\cite{boffi2013mixed}. 
	To provide it, we need an additional bilinear form $s(\cdot,\cdot)$, which is a symmetric positive semi-definite bilinear form on $L^2(\Omega)\times L^2(\Omega)$. 
	Furthermore, we assume that  the following inf-sup condition is satisfied 
	\begin{equation}\label{ineq:inf-sup-b}
		\sup_{0\neq\bm{v}_h\in \bm{V}_h} \frac{b(\bm{v}_h,q_h)}{\|\nabla\bm{v}_h\|}  +s(q_h,q_h)^{1/2} \geq\delta_0 \|q_h\| \quad \forall q_h\in Q_h,
	\end{equation}
	where $\delta_0>0$ is a constant that does not depend on the mesh parameter $h$.
	Suitable definitions of  $s(g_h,q_h)$ can be found, for example, in \cite{burman2008pressure,garcia2021symmetric}.  
	For the analysis that follows,  a particular choice of $s(\cdot,\cdot)$ is not important.

	A finite element formulation for the Oseen problem reads: Find $ \{\bm{u}_h,p_h\}\in \bm{X}=\bm{V}_h \times Q_h$ such that
	\begin{align*}
		a(\bm{u}_h,\bm{v}_h)+  a_s(\bm{u}_h,\bm{v}_h) +b(\bm{v}_h,p_h) &=\langle\bm{f},\bm{v}_h\rangle\quad \forall \bm{v}_h\in \bm{V}_h, \\
		b(\bm{u}_h,q_h)-\frac{1}{\nu+\gamma}s(p_h,q_h) & =0 \quad \forall q_h\in Q_h,
	\end{align*}
	where $a_s(\bm{u},\bm{v})=\gamma({\rm div}\bm{u}, {\rm div}\bm{v})$ and $\gamma\ge0$ is a grad-div stabilization parameter. 
	We assume that the flow is laminar and the mesh is sufficiently fine so that no additional closure or advection stabilization terms are needed.

	It is convenient to rewrite the finite element problem using the cumulative form 
	\begin{equation*}
		\mathcal{L}(\bm{w}_h,g_h;\bm{v}_h,q_h) = a(\bm{w}_h,\bm{v}_h)+ a_s(\bm{w}_h,\bm{v}_h) +b(\bm{v}_h,g_h)+b(\bm{w}_h,q_h)-\frac{1}{\nu+\gamma}s(g_h,q_h)
	\end{equation*}
	as follows: 
	Find $ \{\bm{u}_h,p_h\}\in \bm{X}$ such as 
	\begin{equation}\label{eq:bilinear-L}
		\mathcal{L}(\bm{u}_h,p_h;\bm{v}_h,q_h) =\langle\bm{f},\bm{v}_h\rangle  \quad  \forall  \{\bm{v}_h,q_h\}\in \bm{X}.
	\end{equation}

	For notation simplicity, let $x=\{\bm{w}_h,g_h\}$ and $y=\{\bm{v}_h,q_h\}$. Based on the symmetric part of $\mathcal{L}(x;y)$, we introduce the following products and norms on $\bm{X}$: 
	\begin{align*}
		(x,y)_{\mathcal{L}} = \nu(\nabla \bm{w}_h, \nabla \bm{v}_h) +\gamma ({\rm div}\bm{w}_h, {\rm div}\bm{v}_h) + \frac{1}{\nu+\gamma}\left[(g_h,q_h) +s(g_h,q_h)\right],\quad  \| x\|_{\mathcal{L}}^2 = (x,x)_{\mathcal{L}}.
	\end{align*}
	Further in the text, $a \gtrsim b$ means that there exists a constant $c$ independent of mesh size and other problem parameters such that  $a\geq cb$. Obviously, $a \lesssim b$ is defined as $b \gtrsim a$.
	To avoid nonessential technical details, it is not restrictive to assume for the rest of the paper that
	\begin{equation}\label{ass1}
	\|\bm{a}\|_{{L^\infty(\Omega)}}=1\quad\text{and}\quad\nu\lesssim 1,~~ \gamma\lesssim \rev{\nu^{-1}}.
	\end{equation}
  
	\begin{lemma}\label{thm:conti-stabi-conditions}
		The bilinear form $\mathcal{L}(x;y) $ satisfies the following continuity and stability property:
		\begin{equation}\label{eq:conti-L}
			\sup_{0\neq x\in \bm{X}}\sup_{0\neq y\in \bm{X}} \frac{\mathcal{L}(x;y)}{\|x\|_{\mathcal{L}} \|y\|_{\mathcal{L}}} \leq c_1,
		\end{equation}
		and
		\begin{equation}\label{eq:stabi-L}
			\inf_{0\neq x\in \bm{X}}\sup_{0\neq y\in \bm{X}} \frac{\mathcal{L}(x;y)}{\|x\|_{\mathcal{L}} \|y\|_{\mathcal{L}}} \geq c_2,
		\end{equation}
		with some positive mesh-independent constants
		\begin{equation*}
			c_1\lesssim (1+\nu^{-1}), \quad c_2 \gtrsim \nu(\nu+\gamma).
		\end{equation*}
	\end{lemma}
	\begin{proof} We shall repeatedly use  the Poincar\'e–Friedrichs inequality:
		\begin{equation*}
			\|{\bm{v}}  \|\leq c_f \|\nabla {\bm{v}}\| \quad \forall {\bm{v}}  \in H^1_0(\Omega)^d.
		\end{equation*}

		For  $x=\{\bm{w}_h,g_h\}, y=\{\bm{v}_h,q_h\}$, we have
		\begin{align*}
			\mathcal{L}(x;y) &= a(\bm{w}_h,\bm{v}_h)+ a_s(\bm{w}_h,\bm{v}_h) +b(\bm{v}_h,g_h)+b(\bm{w}_h,q_h)-\tfrac{1}{\nu+\gamma}s(g_h,q_h)\\
			&\leq  a(\bm{w}_h,\bm{w}_h)^{\frac{1}{2}} a(\bm{v}_h,\bm{v}_h)^{\frac{1}{2}}+ \|\bm{a}\|_{L^{\infty}(\Omega)} \|\nabla\bm{w}_h\| \|{\bm{v}}_h  \| + a_s(\bm{w}_h,\bm{w}_h)^{\frac{1}{2}}a_s(\bm{v}_h,\bm{v}_h)^{\frac{1}{2}}  \\ 
			& \quad  + \|{\rm div}\bm{v}_h\| \|g_h\|+\|{\rm div}\bm{w}_h\| \|q_h\|  +\frac{1}{\nu+\gamma}s(g_h,g_h)^{\frac{1}{2}}s(q_h,q_h)^{\frac{1}{2}}\\
			& \leq a(\bm{w}_h,\bm{w}_h)^{\frac{1}{2}} a(\bm{v}_h,\bm{v}_h)^{\frac{1}{2}}+  \frac{c_f}{\nu} \|\sqrt{\nu}\nabla\bm{w}_h\| \|\sqrt{\nu}\nabla {\bm{v}}_h\| + a_s(\bm{w}_h,\bm{w}_h)^{\frac{1}{2}}a_s(\bm{v}_h,\bm{v}_h)^{\frac{1}{2}}  \\ 
			& \quad  + (\nu+\gamma)^{-\frac{1}{2}}\|g_h\| (\gamma^{\frac{1}{2}}\|{\rm div}\bm{v}_h\| +\nu^{\frac{1}{2}}\|\nabla \bm{v}_h\|) +   (\nu+\gamma)^{-\frac{1}{2}}\|q_h\| (\gamma^{\frac{1}{2}}\|{\rm div}\bm{w}_h\|\\
			& \quad +\nu^{\frac{1}{2}}\|\nabla \bm{w}_h\|) +\tfrac{1}{\nu+\gamma}s(g_h,g_h)^{\frac{1}{2}}s(q_h,q_h)^{\frac{1}{2}} \\	 		 	 
			& \leq c_1 \|\bm{w}_h,g_h\|_{\mathcal{L}} \|\bm{v}_h,q_h\|_{\mathcal{L}},
		\end{align*}
		where $c_1=3+\frac{c_f}{\nu}$. This proves \eqref{eq:conti-L}.
		
\rev{We shall also use $ \|\mbox{div}\,{\bm{v}}\|\le \|\nabla {\bm{v}}\| $ inequality, which follows for $\bm{v}\in H^1_0(\Omega)^d$ from
	$\Delta=\nabla\mbox{div}-\nabla\times\nabla\times$ identity and an integration by parts argument.}		
	We introduce the norm $	\|\bm{v}_h\|^2_\ast:= \nu\|\nabla \bm{v}_h\|^2+ \gamma \|{\rm div}\bm{v}_h\|^2$ and notice that 
		\begin{equation*}
			\|\bm{v}_h\|^2_\ast \leq (\nu+\gamma) \|\nabla\bm{v}_h\|^2.
		\end{equation*}
		Thanks to the inf-sup condition \eqref{ineq:inf-sup-b}, for any $g_h$, there exists $\hat{\bm{w}}_h$ such that \[b(\hat{\bm{w}}_h,g_h) \geq \left({\delta_0}\|g_h\| - s(g_h,g_h)^{\frac{1}{2}}\right)\|\nabla\hat{\bm{w}}_h\|.\] We may assume that $\|\nabla \hat{\bm{w}}_h\|=\delta_0\|g_h\|+s(g_h,g_h)^{\frac12}$. 
		For given $x=\{\bm{w}_h,g_h \}$, let us consider $y=\{\bm{v}_h,q_h \}$ with $\bm{v}_h=\bm{w}_h+ \theta \hat{\bm{w}}_h$ and $q_h=-g_h$. Then it holds that
		\begin{align*}
			\mathcal{L}&(x;y) = a(\bm{w}_h,\bm{w}_h+\theta \hat{\bm{w}}_h)+ a_s(\bm{w}_h,\bm{w}_h+\theta \hat{\bm{w}}_h) +b(\bm{w}_h+\theta \hat{\bm{w}}_h,g_h)-b(\bm{w}_h,g_h)+\tfrac{1}{\nu+\gamma}s(g_h,g_h)\\
			&= a(\bm{w}_h,\bm{w}_h)+ a_s(\bm{w}_h,\bm{w}_h)   +a(\bm{w}_h,\theta \hat{\bm{w}}_h)+a_s(\bm{w}_h,\theta \hat{\bm{w}}_h)+b(\theta \hat{\bm{w}}_h,g_h)+\tfrac{1}{\nu+\gamma}s(g_h,g_h) \\
			&=\|\bm{w}_h\|^2_\ast +\tfrac{1}{\nu+\gamma}s(g_h,g_h) +\nu \theta (\nabla \bm{w}_h, \nabla \hat{\bm{w}}_h) +\gamma \theta ({\rm div}\bm{w}_h, {\rm div}\hat{\bm{w}}_h)
			+ \theta (\bm{a} \cdot \nabla \bm{w}_h, \hat{\bm{w}}_h) +b(\theta \hat{\bm{w}}_h,g_h)  \\
			& \geq  \|\bm{w}_h\|^2_\ast +\tfrac{1}{\nu+\gamma}s(g_h,g_h) -\theta \|\bm{w}_h\|_\ast \| \hat{\bm{w}}_h\|_\ast - \theta \|\nabla \bm{w}_h\| \|\hat{\bm{w}}_h\|+
			\theta  \left(\delta_0\|g_h\| - s(g_h,g_h)^{\frac{1}{2}}\right)\|\nabla\hat{\bm{w}}_h\|  \\
			& \geq  \|\bm{w}_h\|^2_\ast +\tfrac{1}{\nu+\gamma}s(g_h,g_h) - \frac14\|\bm{w}_h\|_\ast^2 -\theta^2 \| \hat{\bm{w}}_h\|_\ast^2 - \frac\nu4 \|\nabla \bm{w}_h\|^2  - \frac{\theta^2} {\nu}\|\hat{\bm{w}}_h\|^2+
			\theta   \left(\delta_0^2\|g_h\|^2 - s(g_h,g_h)\right) \\
			& \geq  \frac12\|\bm{w}_h\|^2_\ast +\tfrac{1}{\nu+\gamma}s(g_h,g_h) -\theta^2(\nu+\gamma) \| \nabla\hat{\bm{w}}_h\|^2  - \frac{\theta^2} {\nu}\yh{c_f^2}\|\nabla\hat{\bm{w}}_h\|^2+
			\theta   \left(\delta_0^2\|g_h\|^2 - s(g_h,g_h)\right)  \\
			& \geq  \frac12\|\bm{w}_h\|^2_\ast +(\tfrac{1}{\nu+\gamma}-\theta-2\theta^2(\nu+\gamma) - 2\theta^2 \nu^{-1}\yh{c_f^2})s(g_h,g_h)+
			\theta(1-2\theta(\nu+\gamma) - 2\theta \nu^{-1}\yh{c_f^2})\delta_0^2\|g_h\|^2.
		\end{align*}
		Letting $\theta= \tilde c \nu$	with sufficiently small  $\tilde c\simeq1$ \rev{and using \eqref{ass1}}, we obtain
		\[
		\mathcal{L}(x;y)\gtrsim  \|\bm{w}_h\|^2_\ast+\frac{1}{\nu+\gamma}s(g_h,g_h)+\nu\|g_h\|^2. 
		\]
		This and  $\|\bm{w}_h,g_h\|_{\mathcal{L}}\gtrsim\|\bm{v}_h,q_h\|_{\mathcal{L}}$ prove the result in \eqref{eq:stabi-L}.
	\end{proof}

	\section{Algebraic problem and Schur complement preconditioner}\label{sec:algebraic-preconditioners}
	Let $\{\psi_i\}_{i=1,\dots,n}$ and $\{\phi_i\}_{i=1,\dots,m}$ be bases in $\bV_h$ and $Q_h$, respectively. Define the following matrices: 
	\begin{align*}
		D=\{D_{ij}\},&~ D_{ij}=(\nabla \psi_i, \nabla \psi_j), \quad~  	N=\{N_{ij}\},~ N_{ij}= (\bm{a} \cdot \nabla \psi_j, \psi_i),\\  
		B=\{B_{ij}\}, &~ B_{ij}=-(\phi_j,{\rm div}\psi_i), \quad C=\{C_{ij}\},~ C_{ij}= (\nu+\gamma)^{-1}s(\phi_i,\phi_j), \nonumber \\
		G=\{G_{ij}\}, &~ G_{ij}=({\rm div}\psi_i, {\rm div}\psi_j), \quad A_{\gamma}=\nu D+N+\gamma G,\quad A_s = \frac{1}{2}(A_{\gamma} + A_{\gamma}^T)=\nu D +\gamma G.  \nonumber	
	\end{align*}
	Note that $A_s$ is a symmetric and positive definite matrix.
	
	For a natural $\bu$--$p$ order of unknowns, the finite element method results in a system of algebraic equation with matrices having the block structure:
	\begin{equation}
		\mathcal{K} \yh{x}= \begin{pmatrix}
			A_\gamma & ~B^T\\
			B &  -C
		\end{pmatrix}\yh{x} =b,
	\end{equation}
	where $A_{\gamma}\in\mathbb{R}^{n\times n}, C\in \mathbb{R}^{m\times m}$, and $B\in \mathbb{R}^{m\times n}$.
	
	We are also interested in the Schur complement for $\mathcal{K}$:
	\[
	S= B A_{\gamma}^{-1}B^T+C.
	\]
	
	Let $M$ be the pressure mass matrix associated with element space $Q_h$. We introduce the following preconditioner for $S$: 
	\begin{equation}\label{Sh}
		\widehat{S}=(\nu+\gamma)^{-1} M + C,
	\end{equation}
	and a block diagonal matrix 
	\begin{equation}\label{Mm}
		\mathcal{M} =\begin{pmatrix}
			A_s  & 0\\
			0 &   \widehat{S}
		\end{pmatrix}.
	\end{equation}
	Note that $\mathcal{M}$ is not a preconditioner for $\mathcal{K}$ we are interested in. This matrix  is introduced to define convenient inner products to work with.
	We shall adopt some standard notations:
	Let $\langle\cdot,\cdot\rangle$	be the Euclidean inner product and $\|\cdot\|_{\ell^2}=\langle\cdot,\cdot\rangle^{\frac{1}{2}}$, also
	$\langle \cdot,\cdot\rangle_{F}=\langle F \cdot,\cdot\rangle$ and $\|\cdot\|_F=\langle \cdot,\cdot\rangle_{F}^{\frac12}$ denote an inner product and norm on $\mathbb{R}^k$ given a  positive definite symmetric matrix $F\in \mathbb{R}^{k\times k}$.

	\begin{theorem} The stability and continuity properties from Lemma~\ref{thm:conti-stabi-conditions} are equivalent to the following conditions for the matrix $\mathcal{K}$:
		\begin{equation}\label{eq:conti-K}
			\sup_{0\neq x\in\mathbb{R}^{n+m}}\sup_{0\neq y\in\mathbb{R}^{n+m}} \frac{\langle \mathcal{K} x,y\rangle}{ \|x\|_{\mathcal{M}} \|y\|_{\mathcal{M}}} \leq c_1,
		\end{equation}
		and 
		\begin{equation}\label{eq:stabi-K}
			\inf_{0\neq x\in\mathbb{R}^{n+m}}\sup_{0\neq y\in\mathbb{R}^{n+m}}	\frac{\langle \mathcal{K} x,y \rangle}{\|x\|_{\mathcal{M}} \|y\|_{\mathcal{M}}} \geq c_2,
		\end{equation}
		where the constants $c_1$ and $c_2$ are the same as those in \eqref{eq:conti-L} and \eqref{eq:stabi-L}, respectively.
	\end{theorem}

	\subsection{One eigenvalue problem}
	To understand the properties of the Schur complement preconditioner \eqref{Sh}, it is insightful to analyze the eigenvalue problem
	\begin{equation}\label{SSh}
		S \varphi=\lambda\widehat{S}\varphi.  
	\end{equation}
	
	\begin{theorem}\label{thm:Schur-eigs-bound}
		The eigenvalues from \eqref{SSh} satisfy the following bounds
		\begin{equation*}
			\nu(\nu+\gamma)  \lesssim \Re(\lambda), \quad  |\lambda| \leq 1.
		\end{equation*}
	\end{theorem}

	\begin{proof}
		The proof partially follows \cite{olshanskii2022recycling}. 
		For $\varphi \in \mathbb{R}^m$, the corresponding finite element function is $\varphi_h \in Q_h$. For $w \in \mathbb{R}^n$, the corresponding  finite element function is $\bm{w}_h \in \bm{V}_h$.  According to Bendixson theorem,
		\begin{equation*}
			\inf_{0\neq \varphi\in \mathbb{R}^m} \frac{\left\langle \frac{1}{2} (S+S^T)\varphi,\varphi \right\rangle}{\langle \hat{S}\varphi,\varphi\rangle} \leq \Re(\lambda).
		\end{equation*}
		Note that $\left\langle \frac{1}{2} (S+S^T)\varphi,\varphi \right\rangle=\langle S\varphi,\varphi \rangle$. For a given $\varphi$, define $w$ by $A_{\gamma}w=B^T\varphi$. Then, 
		\begin{align*}
			\langle S\varphi,\varphi \rangle&=\langle Bw,\varphi\rangle + \langle C\varphi,\varphi\rangle,  \\
			A_{\gamma}w &=B^T\varphi.
		\end{align*}
		We rewrite the above formulas in the finite element notations:
		\begin{align}
			\langle S\varphi,\varphi\rangle &=b(\bm{w}_h,\varphi_h) + \frac1{\nu+\gamma}s(\varphi_h,\varphi_h), \label{eq:FEMs}\\
			b(\bm{v}_h,\varphi_h)& = a (\bm{w}_h, \bm{v}_h) + a_s(\bm{w}_h,\bm{v}_h),\quad \forall \bm{v}_h \in  \bm{V}_h.\label{eq:FEMb}
		\end{align}
		Taking $\bm{v}_h= \bm{w}_h$ in \eqref{eq:FEMb} and substituting $b(\bm{w}_h,\varphi_h)$ into \eqref{eq:FEMs} leads to
		\begin{equation*}
			\langle S\varphi,\varphi\rangle  =  a (\bm{w}_h, \bm{w}_h) + a_s(\bm{w}_h,\bm{w}_h)+ \frac1{\nu+\gamma}s(\varphi_h,\varphi_h).
		\end{equation*}

		From \eqref{ineq:inf-sup-b} it follows that there exists $\hat{\bm{v}}_h \in \bm{V}_h$ with $\|\nabla \hat{\bm{v}}_h \|=1$ such that 
		\begin{equation*}
			\delta \|\varphi_h\|^2 \leq  b(\hat{\bm{v}}_h ,\varphi_h)^2 + s( \varphi_h,\varphi_h).
		\end{equation*} 
		Recall that $\|\bm{a}\|_{L^{\infty}(\Omega)}=1$  and  ${\rm div}\bm{a}=0$. Taking $\hat{\bm{v}}_h$ as a test function in \eqref{eq:FEMb} and using Cauchy-Schwarz and Poincar\'e inequalities,  we get 
		\begin{align*}
			\delta \|\varphi_h\|^2 &\leq  ( a(\bm{w}_h, \hat{\bm{v}}_h) + a_s(\bm{w}_h,\hat{\bm{v}}_h))^2 + s( \varphi_h,\varphi_h) \\
			&\leq  \left( a(\bm{w}_h,\bm{w}_h)^{\frac{1}{2}} a(\hat{\bm{v}}_h, \hat{\bm{v}}_h)^{\frac{1}{2}} +\|\bm{a}_h\|_{L^{\infty}(\Omega)} \|\nabla\bm{w}_h\| \|\hat{\bm{v}}_h  \| +\gamma \|{\rm div}\bm{w}\| \|{\rm div}\hat{\bm{v}}_h \| \right)^2  + 
			s( \varphi_h,\varphi_h)\\
			&\leq \left( a(\bm{w}_h,\bm{w}_h)^{\frac{1}{2}} a(\hat{\bm{v}}_h, \hat{\bm{v}}_h)^{\frac{1}{2}} + c_f\|\bm{a}_h\|_{L^{\infty}(\Omega)} \|\nabla\bm{w}_h\| \| \nabla \hat{\bm{v}}_h  \| +\gamma \|{\rm div}\bm{w}_h\| \| \nabla \hat{\bm{v}}_h \| \right)^2 +  s( \varphi_h,\varphi_h)\\
			& \leq \left(  \sqrt{\nu}a(\bm{w}_h,\bm{w}_h)^{\frac{1}{2}}  +  \frac{c_f}{\sqrt{\nu}} a(\bm{w}_h,\bm{w}_h)^{\frac{1}{2}}  +\gamma \|{\rm div}\bm{w}_h\|\right)^2 +s( \varphi_h,\varphi_h),	 	
		\end{align*}
		where in the last inequality we used 
		\begin{equation*}
			a(\hat{\bm{v}}_h, \hat{\bm{v}}_h) =\nu \|\nabla\hat{\bm{v}}_h\|^2=\nu, \quad \|{\rm div}\hat{\bm{v}}_h \| \leq \|\nabla \hat{\bm{v}}_h \|=1, \quad \|\nabla \bm{w}_h\| =\frac{1}{\sqrt{\nu}}a(\bm{w}_h,\bm{w}_h)^{\frac{1}{2}}.
		\end{equation*}
		Thus, we get
		\begin{align*}
			\delta \|\varphi_h\|^2  &\leq 2 \left( (\sqrt{\nu}+ c_f/\sqrt{\nu})^2a(\bm{w}_h,\bm{w}_h)+\gamma^2 \|{\rm div}\bm{w}_h\|^2 \right)+s( \varphi_h,\varphi_h) \\
			& \leq 2 \left( 2(\nu+ c_f^2 \nu^{-1}) a(\bm{w}_h,\bm{w}_h)+ \rev{\gamma} a_s(\bm{w}_h,\bm{w}_h) \right)+ (\nu+\gamma)\frac{1}{\nu+\gamma} s(\varphi_h,\varphi_h)\\
			& \leq 4(\nu+1+\rev{\gamma}+ c_f^2 \nu^{-1}) \langle S\varphi,\varphi\rangle \lesssim (1+\nu^{-1}) \langle S\varphi,\varphi\rangle.
		\end{align*}
		From this, assumptions \eqref{ass1} and the definition of $S$, it  follows that
		\begin{equation*}
			\nu (\nu+\gamma) \langle ((\nu+\gamma)^{-1} M +C)\varphi, \varphi\rangle \lesssim  \langle S\varphi,\varphi\rangle,
		\end{equation*}
		which can be rewritten as
		\begin{equation}\label{ineq:S-larger-than-hats}
			\nu (\nu+\gamma)   \lesssim \frac{\langle S\varphi,\varphi\rangle}{ \langle \widehat{S}\varphi, \varphi\rangle}.
		\end{equation}
		Next, we show the bound on $|\lambda|$.	
		\begin{equation*}
			|\lambda| \leq \|\widehat{S}^{-\frac{1}{2}} S \widehat{S}^{-\frac{1}{2}}\|=\sup_{0\neq \varphi,q\in\mathbb{R}^m} \frac{\langle \widehat{S}^{-\frac{1}{2}} S \widehat{S}^{-\frac{1}{2}}\varphi,q\rangle}{ \|\varphi\|_{\ell^2}\|q\|_{\ell^2}} = \sup_{0\neq \varphi,q\in\mathbb{R}^m} \frac{\langle S\varphi,q\rangle}{ \|\widehat{S}^{\frac{1}{2}}\varphi \|_{\ell^2}\|\widehat{S}^{\frac{1}{2}}q\|_{\ell^2}}.
		\end{equation*}
		In the finite element notations, we rewrite
		\begin{align*}
			\langle S\varphi,q\rangle &=b(\bm{w}_h,q_h) + \yh{\frac1{\nu+\gamma}}s(\varphi_h,q_h),  \\
			b(\bm{v}_h,\varphi_h)& = a (\bm{w}_h, \bm{v}_h) + a_s(\bm{w}_h,\bm{v}_h),\quad \forall \bm{v}_h \in  \bm{V}_h. 
		\end{align*}
		Using Cauchy-Schwarz inequality for the above equations, we get
		\begin{align}
			\langle S\varphi,q\rangle & \leq  \|{\rm div}\bm{w}_h\| \|q_h\|+ \frac1{\nu+\gamma}s(\varphi_h,\varphi_h)^{\frac{1}{2}} s(q_h,q_h)^{\frac{1}{2}}, \label{eq:Spq}  \\
			a(\bm{w}_h, \bm{w}_h) + \gamma \|{\rm div}\bm{w}_h\|^2 &=b(\bm{w}_h,\varphi_h)\leq \|{\rm div}\bm{w}_h \| \|\varphi_h\|. \label{eq:auu}
		\end{align}
		Using $\|{\rm div}\bm{w}_h\| \leq \|\nabla \bm{w}_h\|$ and $ a(\bm{w}_h, \bm{w}_h)=\nu \|\nabla\bm{w}_h\|^2$,  we obtain from  \eqref{eq:auu} that
		\begin{equation*}
			(\nu+\gamma) \|{\rm div}\bm{w}_h\| \leq \|\varphi_h\|. 
		\end{equation*}
		Using the above inequality, we can rewrite \eqref{eq:Spq} as
		\begin{align*}
			\langle S\varphi,q\rangle &\leq  \frac{1}{\nu+\gamma}	\left[\|\varphi_h\| \|q_h\| + s(\varphi_h,\varphi_h)^{\frac{1}{2}} s(q_h,q_h)^{\frac{1}{2}}\right]\\
			&\leq  	\|\frac{1}{\sqrt{\nu+\gamma}} \varphi_h\| \|\frac{1}{\sqrt{\nu+\gamma} }q_h\| + \frac1{\nu+\gamma}s(\varphi_h,\varphi_h)^{\frac{1}{2}} s(q_h,q_h)^{\frac{1}{2}}\\
			&\leq  \|\widehat{S}^{\frac{1}{2}}\varphi\|_{\ell^2}\|\widehat{S}^{\frac{1}{2}}q\|_{\ell^2}.
		\end{align*}
		It follows that 
		\begin{equation}\label{inq:S-hatS-RQ}
			\frac{\langle S\varphi,q\rangle}{ \|\widehat{S}^{\frac{1}{2}}\varphi\|_{\ell^2}\|\widehat{S}^{\frac{1}{2}}q\|_{\ell^2}}\leq 1,
		\end{equation}
		which is the desired result.
	\end{proof}
	Theorem \ref{thm:Schur-eigs-bound} shows bounds similar to those found in the literature on the AL approach with inf-sup stable finite elements. In particular, we see that all the eigenvalues of the preconditioned Schur complement are bounded on the right half of the complex plane, independently of the mesh parameter. Moreover, the bounds improve as $\gamma$ increases.

	\section{Block preconditioner and convergence analysis}\label{sec:field-of-values2}

	For a matrix $\mathcal{B} \in \mathbb{R}^{n+m}$,  denote 
	\begin{equation*}
		\mu(\mathcal{B})= \inf_{0\neq z\in \mathbb{R}^{n+m}} \frac{\langle \mathcal{B}z, z\rangle_{\mathcal{M}^{-1}} }{ \langle z,z \rangle_{\mathcal{M}^{-1}}}.
	\end{equation*}
	One can show that $	\mu(\mathcal{B})	\mu(\mathcal{B}^{-1})\le 1$. The following lemma provides an upper bound for the convergence of preconditioned GMRES:
	\begin{lemma}[\cite{starke1997field}]\label{Lres}
		The residual of the preconditioned GMRES method for $\mathcal{K} x =b$ with the preconditioner $ \mathcal{F}$  satisfies
		\begin{equation}
			\frac{\| r_k\|_{\mathcal{M}^{-1}}} {\| r_0\|_{\mathcal{M}^{-1}}}  \leq \left(1- \mu(\mathcal{K} \mathcal{F}^{-1}) \mu(\mathcal{F}\mathcal{K}^{-1})\right)^{k/2} .
		\end{equation}
	\end{lemma}

An `ideal' block preconditioner for $\mathcal{K}$ is given by 
	\begin{equation}\label{eq:ideal-precon}
		\mathcal{P} =\begin{pmatrix}
			A_{\gamma}  & B^T\\
			0 &  -S
		\end{pmatrix}.	
	\end{equation}
For this preconditioner, one computes
	\begin{equation*}
		\mathcal{K}^{-1} =\begin{pmatrix}
			A^{-1}_{\gamma} - A^{-1}_{\gamma}B^T S^{-1} B A^{-1}_{\gamma} &   A^{-1}_{\gamma}B^T S^{-1}\\
			S^{-1} BA^{-1}_{\gamma}  &-S^{-1}
		\end{pmatrix}
	\end{equation*} 
	and
	\begin{equation*}
		\mathcal{K}\mathcal{P}^{-1} = \begin{pmatrix}
			I  &0\\
			B A_{\gamma}^{-1} &  I
		\end{pmatrix}, \quad
		\mathcal{P}\mathcal{K}^{-1}  = \begin{pmatrix}
			I  &0\\
			-  B A_{\gamma}^{-1} &    I
		\end{pmatrix}.
	\end{equation*}
 When GMRES is applied to the preconditioned system with preconditioner $\mathcal{P}^{-1}$, it converges in 2 iterations~\cite{elman2014finite}.  
 
 In practice, it is too expensive to use the ideal preconditioner.
A more practical choice consists in replacing $S$ with $\widehat{S}$ in it. 
Therefore the preconditioner we analyze here is given by 
	\begin{equation}\label{eq:prec-inexact}
		\mathcal{\widehat{P}} =\begin{pmatrix}
			A_{\gamma}  & B^T\\
			0 &  -\widehat{S}
		\end{pmatrix},
	\end{equation}
with $\widehat{S}$ from \eqref{Sh}.	

	It can be shown that 
	\begin{equation*}
		\mathcal{K} \mathcal{\widehat{P}}^{-1} = \begin{pmatrix}
			I  &0\\
			B A_{\gamma}^{-1} &   S\widehat{S}^{-1}
		\end{pmatrix}, \quad  \mathcal{\widehat{P}}\mathcal{K}^{-1}  = \begin{pmatrix}
			I  &0\\
			- \widehat{S} S^{-1}B A_{\gamma}^{-1} &    \widehat{S} S^{-1}
		\end{pmatrix}.
	\end{equation*}
	
	\begin{lemma}
		For any $q$, we have 
		\begin{equation}\label{ineq:d1}
			\nu(\nu+\gamma) \langle \widehat{S}^{-1}q, q\rangle \lesssim  \langle \widehat{S}^{-1} S \widehat{S}^{-1}q,q\rangle
		\end{equation} 	 	
		and 
		\begin{equation}\label{ineq:hatSinv-S-inv}
			\langle \widehat{S}^{-1}q,q\rangle  \leq \langle S^{-1}q,q\rangle.  
		\end{equation}
	\end{lemma}	
	\begin{proof}
		From \eqref{ineq:S-larger-than-hats}, we have
		\begin{equation*}
			\nu(\nu+\gamma) \langle \widehat{S}g, g\rangle \leq  \langle Sg,g\rangle.
		\end{equation*}
		Setting $g=\widehat{S}^{-1}q$ in the above inequality leads to \eqref{ineq:d1}.
		
		Next, we show the result in \eqref{ineq:hatSinv-S-inv}.  Consider $q=\widehat{S}\hat{q}$. Then,
		\begin{equation*}
			\langle \widehat{S}^{-1}q,q\rangle =  \langle \widehat{S}\hat{q},\hat{q}\rangle =\langle ((\gamma +\nu)^{-1}M+C ) \hat{q}, \hat{q}\rangle= (\gamma +\nu)^{-1}\left(\|\hat{q}_h\|^2 +s(\hat{q}_h,\hat{q}_h)\right),
		\end{equation*}
		and 
		\begin{equation*}
			\langle S^{-1}q, q\rangle =	\langle S^{-1}\hat{S}\hat{q} ,\widehat{S}\hat{q}\rangle= \langle g, \widehat{S}\hat{q}\rangle = (\gamma +\nu)^{-1} \left((g_h,\hat{q}_h)  +s(g_h,\hat{q}_h)\right),
		\end{equation*}
		where $g$ solves
		\begin{align*}
			A_{\gamma} w+B^Tg =0\quad\text{with}\quad
			-Bw+Cg= (\gamma +\nu)^{-1}M \yh{\hat{q}+C\hat{q}}.
		\end{align*}
		The finite element forms of the above equations are
		\begin{align}
			\nu(\nabla \bm{w}_h, \nabla \bm{v}_h) +\gamma ({\rm div}\bm{w}_h, {\rm div}\bm{v}_h) +  (\bm{a} \cdot \nabla \bm{w}_h, \bm{v}_h)  &=(g_h, {\rm div}\bm{v}_h), \label{eq:uv0-form}\\
			({\rm div}\bm{w}_h, \tau_h)  +\frac1{\nu+\gamma}s(g_h,\tau_h) &= \frac1{\nu+\gamma}\left[(\hat{q}_h,\tau_h) +s(\hat{q}_h,\tau_h)\right], \label{eq:tau-h-form}
		\end{align}
	for all  $ \bm{v}_h\in\bV_h$, $\tau_h\in Q_h$.	
		Taking $\tau_h=g_h$  in \eqref{eq:tau-h-form} and $\bm{v}_h=\bm{w}_h$ in \eqref{eq:uv0-form}, we have
		\begin{equation*}
			\langle S^{-1}q, q\rangle= ({\rm div}\bm{w}_h, g_h)  +\frac1{\nu+\gamma}s(g_h,g_h)= \nu \|\nabla \bm{w}_h\|^2 +\gamma \|{\rm div}\bm{w}_h\|^2+\frac1{\nu+\gamma}s(g_h,g_h).
		\end{equation*}
		Now letting $\tau_h=\hat{q}_h$  in \eqref{eq:tau-h-form}, we obtain
		\begin{align*}
			\langle \widehat{S} \hat{q},\hat{q}\rangle &=({\rm div}\bm{w}_h, \hat{q}_h)  +\tfrac1{\nu+\gamma}s(g_h,\hat{q}_h)\\
			&\leq \|{\rm div}\bm{w}_h\| \|\hat{q}_h\|  +\tfrac1{\nu+\gamma}s(g_h, g_h)^{\frac{1}{2}} s(\hat{q}_h,\hat{q}_h)^{\frac{1}{2}}\\
			&\leq \left ((\gamma+\nu)\|{\rm div}\bm{w}_h\|^2+  \tfrac1{\nu+\gamma}s(g_h, g_h)\right)^{\frac{1}{2}}  \tfrac1{\nu+\gamma}\left( \|\hat{q}_h\|^2+s(\hat{q}_h,\hat{q}_h) \right)^{\frac{1}{2}},\\
			&=\left ((\gamma+\nu)\|{\rm div}\bm{w}_h\|^2+  \tfrac1{\nu+\gamma}s(g_h, g_h)\right)^{\frac{1}{2}} \langle \widehat{S} \hat{q},\hat{q}\rangle^{\frac{1}{2}}.
		\end{align*}
		It follows that 
		\begin{equation*}
			\begin{split}
				\langle \widehat{S}^{-1}q,q\rangle &=\langle \widehat{S} \hat{q},\hat{q}\rangle \leq (\gamma+\nu)\|{\rm div}\bm{w}_h\|^2+  \frac{s(g_h, g_h)}{\nu+\gamma} \leq \nu \|\nabla \bm{w}_h\|^2 +\gamma\|{\rm div}\bm{w}_h\|^2+ \frac1{\nu+\gamma} s(g_h, g_h)\\
				&=\langle S^{-1}q, q\rangle, 
			\end{split}
		\end{equation*}
		which is the desired result.
	\end{proof}
	
		Following the technique used in \cite{klawonn1999block}, we prove the field-of-values bounds for our preconditioner. 

	\begin{theorem}\label{thm:inexact-second-block}
		The field-of-values of $\mathcal{K}$ with the preconditioner $\mathcal{\hat{P}}$ from \eqref{eq:prec-inexact} satisfy the following bounds:
		\begin{equation*}
			\mu(\mathcal{K}\mathcal{\widehat{P}}^{-1})= \inf_{v\in \mathbb{R}^{n}, q\in \mathbb{R}^{m}} \frac{ \left\langle \begin{pmatrix}
					I  &0\\
					B A_{\gamma}^{-1} &  S\widehat{S}^{-1}
				\end{pmatrix} \begin{pmatrix}
					v   \\ q \end{pmatrix},  \begin{pmatrix}
					v   \\ q \end{pmatrix} \right\rangle_{\mathcal{M}^{-1}} }{ \left\langle \begin{pmatrix}
					v   \\ q \end{pmatrix},  \begin{pmatrix}
					v   \\ q \end{pmatrix} \right\rangle_{\mathcal{M}^{-1}}} \gtrsim \nu(\nu+\gamma),  
		\end{equation*}
		and 
		\begin{equation*}
			\mu(\mathcal{\widehat{P}}\mathcal{K}^{-1})= \inf_{v\in \mathbb{R}^{n}, q\in \mathbb{R}^{m}} \frac{ \left\langle \begin{pmatrix}
					I  &0\\
					-\widehat{S}S^{-1}B A_{\gamma}^{-1} & \widehat{S}S^{-1}
				\end{pmatrix} \begin{pmatrix}
					v   \\ q \end{pmatrix},  \begin{pmatrix}
					v   \\ q \end{pmatrix} \right\rangle_{\mathcal{M}^{-1}} }{ \left\langle \begin{pmatrix}
					v   \\ q \end{pmatrix},  \begin{pmatrix}
					v   \\ q \end{pmatrix} \right\rangle_{\mathcal{M}^{-1}}}\geq \frac{1}{2}. 
		\end{equation*}
		Hence the residual norms of the preconditioned GMRES  \eqref{eq:prec-inexact} satisfy 
		\begin{equation}
			\| r_k\|_{\mathcal{M}^{-1}} \leq \left(1-\tilde c \nu(\nu+\gamma)\right)^{k/2} \| r_0\|_{\mathcal{M}^{-1}}=\yh{\varrho^{k}}\| r_0\|_{\mathcal{M}^{-1}},
		\end{equation}	
		where $\tilde c\simeq O(1)$ is independent of the mesh size and problem parameters. 
	\end{theorem}
	\begin{proof}
		For arbitrary $[v;q]\neq 0$, we need to show that 
		\begin{equation*}
			\mu(\mathcal{K}\mathcal{\widehat{P}}^{-1})= \inf_{v\in \mathbb{R}^{n}, q\in \mathbb{R}^{m}} \frac{ \left\langle \begin{pmatrix}
					A^{-1}_s  &0\\
					\widehat{S}^{-1}B A_{\gamma}^{-1} &  \widehat{S}^{-1}S\widehat{S}^{-1}
				\end{pmatrix} \begin{pmatrix}
					v   \\ q \end{pmatrix},  \begin{pmatrix}
					v   \\ q \end{pmatrix} \right\rangle} { \left\langle \begin{pmatrix}
					A^{-1}_s & 0   \\ 0& \widehat{S}^{-1}\end{pmatrix} \begin{pmatrix}
					v   \\ q \end{pmatrix},  \begin{pmatrix}
					v   \\ q \end{pmatrix} \right\rangle} \gtrsim \nu(\nu+\gamma),  
		\end{equation*}
		and 
		\begin{equation*}
			\mu(\mathcal{\widehat{P}}\mathcal{K}^{-1})= \inf_{v\in \mathbb{R}^{n}, q\in \mathbb{R}^{m}} \frac{ \left\langle \begin{pmatrix}
					A^{-1}_s  &0\\
					-S^{-1}B A_{\gamma}^{-1} &  S^{-1}
				\end{pmatrix} \begin{pmatrix}
					v   \\ q \end{pmatrix},  \begin{pmatrix}
					v   \\ q \end{pmatrix} \right\rangle}{ \left\langle \begin{pmatrix}
					A^{-1}_s & 0   \\ 0& \widehat{S}^{-1} \end{pmatrix}\begin{pmatrix}
					v   \\ q \end{pmatrix},  \begin{pmatrix}
					v   \\ q \end{pmatrix} \right\rangle} \geq  \frac{1}{2}.
		\end{equation*}
		With the help of \eqref{ineq:d1}, we obtain
		\begin{align} 
			\mu(\mathcal{K}\mathcal{\widehat{P}}^{-1})&= \inf_{v\in \mathbb{R}^{n}, q\in \mathbb{R}^{m}} \frac{ \left\langle \begin{pmatrix}
					A^{-1}_s  &0\\
					\widehat{S}^{-1}B A_{\gamma}^{-1} &  \widehat{S}^{-1}S\widehat{S}^{-1}
				\end{pmatrix} \begin{pmatrix}
					v   \\ q \end{pmatrix},  \begin{pmatrix}
					v   \\ q \end{pmatrix} \right\rangle} { \left\langle \begin{pmatrix}
					A^{-1}_s & 0   \\ 0& \widehat{S}^{-1}\end{pmatrix} \begin{pmatrix}
					v   \\ q \end{pmatrix},  \begin{pmatrix}
					v   \\ q \end{pmatrix} \right\rangle} \nonumber\\
			&\gtrsim \nu(\nu+\gamma)\inf_{v\in \mathbb{R}^{n}, q\in \mathbb{R}^{m}} \frac{ \left\langle \begin{pmatrix}
					A^{-1}_s  &0\\
					\widehat{S}^{-1}B A_{\gamma}^{-1} &   \widehat{S}^{-1}S\widehat{S}^{-1}
				\end{pmatrix} \begin{pmatrix}
					v   \\ q \end{pmatrix},  \begin{pmatrix}
					v   \\ q \end{pmatrix} \right\rangle} { \left\langle \begin{pmatrix}
					A^{-1}_s & 0   \\ 0&  \widehat{S}^{-1}S\widehat{S}^{-1} \end{pmatrix} \begin{pmatrix}
					v   \\ q \end{pmatrix},  \begin{pmatrix}
					v   \\ q \end{pmatrix} \right\rangle}.\label{eq:inexact-KP}
		\end{align}
		Using \eqref{ineq:hatSinv-S-inv}, we get
		\begin{align}
			\mu(\mathcal{\widehat{P}}\mathcal{K}^{-1})&= \inf_{v\in \mathbb{R}^{n}, q\in \mathbb{R}^{m}} \frac{ \left\langle \begin{pmatrix}
					A^{-1}_s  &0\\
					-S^{-1}B A_{\gamma}^{-1} &  S^{-1}
				\end{pmatrix} \begin{pmatrix}
					v   \\ q \end{pmatrix},  \begin{pmatrix}
					v   \\ q \end{pmatrix} \right\rangle}{ \left\langle \begin{pmatrix}
					A^{-1}_s & 0   \\ 0& \widehat{S}^{-1} \end{pmatrix}\begin{pmatrix}
					v   \\ q \end{pmatrix},  \begin{pmatrix}
					v   \\ q \end{pmatrix} \right\rangle}  \nonumber\\
			&\geq  \inf_{v\in \mathbb{R}^{n}, q\in \mathbb{R}^{m}} \frac{ \left\langle \begin{pmatrix}
					A^{-1}_s  &0\\
					-S^{-1}B A_{\gamma}^{-1} &  S^{-1}
				\end{pmatrix} \begin{pmatrix}
					v   \\ q \end{pmatrix},  \begin{pmatrix}
					v   \\ q \end{pmatrix} \right\rangle}{ \left\langle \begin{pmatrix}
					A^{-1}_s & 0   \\ 0& S^{-1} \end{pmatrix}\begin{pmatrix}
					v   \\ q \end{pmatrix},  \begin{pmatrix}
					v   \\ q \end{pmatrix} \right\rangle}. \label{eq:inexact-PK}
		\end{align}
		To estimate the quantities in the right hand sides of \eqref{eq:inexact-KP} and \eqref{eq:inexact-PK}, it is sufficient to find the lower bounds for the generalized eigenvalues of 
		\begin{equation*} 
			\begin{pmatrix}
				H_{11}  &  H_{21}^T    \\
				H_{21}  &  H_{22} 
			\end{pmatrix} \begin{pmatrix}
				v   \\ q \end{pmatrix} =\lambda  \begin{pmatrix}
				H_{11}  &  0    \\
				0  &  H_{22} 
			\end{pmatrix}\begin{pmatrix}
				v   \\ q \end{pmatrix}, 
		\end{equation*}
		with $H_{11} =A_s^{-1}$  and 
		\begin{equation*}
			\begin{cases}
				H_{21}&= \frac{1}{2}\widehat{S}^{-1} B A_{\gamma}^{-1}, \quad  H_{22}= \frac{1}{2}\big(\widehat{S}^{-1}S\widehat{S}^{-1}+ \widehat{S}^{-1}S^T\widehat{S}^{-1}\big) \,\, \,\text{for} \,\, \eqref{eq:inexact-KP}, \\
				H_{21}&= -\frac{1}{2}S^{-1} B A_{\gamma}^{-1}, \quad H_{22}= \frac{1}{2} (S^{-1}+S^{-T})\,\,\, \text{for} \,\,\eqref{eq:inexact-PK}.
			\end{cases}
		\end{equation*}
		Straightforward computations reveal the identity
		\begin{equation*}
			(\lambda-1)^2 =\frac{\langle H_{11}^{-1}H_{21}^Tq, H_{21}^Tq\rangle}{\langle H_{22}q,q\rangle}.
		\end{equation*}
To estimate the right-hand side in	\eqref{eq:inexact-KP}, one sets $g=\widehat{S}^{-1} q$ in the above quantity leading to
		\begin{align*}
			(\lambda-1)^2 &=\frac{1}{4} \frac{\langle A_s A_{\gamma}^{-T} B^T \widehat{S}^{-1} q, A_{\gamma}^{-T} B^T \widehat{S}^{-1} q\rangle }{ \langle \widehat{S}^{-1}S\widehat{S}^{-1}q,q \rangle}
			=\frac{1}{4}\frac{\langle A_s A_{\gamma}^{-T} B^Tg, A_{\gamma}^{-T} B^Tg \rangle  }{ \langle  Sg,g \rangle}\\
			&= \frac{1}{4}\frac{\langle A_s A_{\gamma}^{-T} B^Tg, A_{\gamma}^{-T} B^Tg \rangle}{\langle  BA_{\gamma}^{-1} B^Tg,g\rangle}
			= \frac{1}{4}\frac{\langle A_s A_{\gamma}^{-T} w, A_{\gamma}^{-T}w \rangle}{\langle A_{\gamma}^{-1}w,w\rangle}
			= \frac{1}{4}\frac{\langle A_s w_1, w_1 \rangle}{\langle A_{\gamma} w_1, w_1\rangle}= \frac{1}{4},
		\end{align*}
		where $w=B^Tg$ and $w_1=A_{\gamma}^{-T}w$.

To estimate the right-hand side in	\eqref{eq:inexact-PK}, one sets $g=S^{-1} q$ leading to
		\begin{equation*}
			(\lambda-1)^2 =\frac{1}{4} \frac{\langle A_s A_{\gamma}^{-T} B^T S^{-1} q, A_{\gamma}^{-T} B^T S^{-1} q\rangle}{ \langle S^{-1}q,q \rangle}
			=\frac{1}{4} \frac{\langle A_s A_{\gamma}^{-T} B^Tg, A_{\gamma}^{-T} B^Tg \rangle}{ \langle Sg,g \rangle}
			= \frac{1}{4}.
		\end{equation*}

		Thus, in both cases we have $|1-\lambda|= \frac{1}{2}$. It follows that $\lambda\geq \frac{1}{2}$. Combined with \eqref{eq:inexact-KP} and \eqref{eq:inexact-PK} this proves the  theorem.
	\end{proof}

	\begin{remark}\rm 
		For the analysis of  $\mathcal{\widehat{P}}$, one can replace $\widehat{S}$ by a spectral equivalent operator, $\tilde{S}$, which satisfies
		\begin{equation*}
			\alpha_1 \langle \tilde{S}q,q \rangle \leq  \langle \widehat{S}q,q\rangle \leq \alpha_2 \langle \tilde{S}q, q\rangle,
		\end{equation*}
		where $\alpha_1,\alpha_2>0$, and consider the norm $\|\cdot\|_{\mathcal{M}^{-1}}$, with $\mathcal{M}=\text{diag}\{A_s,\tilde{S}\}$. Then, one obtains similar results to those stated in Theorem~\ref{thm:inexact-second-block}.
	\end{remark}
	
From the above analysis, we see that the simple Schur complement preconditioner \(\widehat{S}\) leads to a mesh-independent solver that is robust with respect to the critical parameter \(\nu\) once \(\gamma\) is \rev{proportional to $\nu^{-1}$}. In \(\mathcal{\widehat{P}}\), we leave open the question of how to approximate the (1,1) block. In the literature, several approaches based on direct LU/ILU factorization, the block upper-triangular part of \(A_{\gamma}\), or the multigrid technique for the (1,1) block have been studied (see \cite{benzi2006augmented, olshanskii2022recycling, borm2012, farrell2021reynolds}). Regarding field-of-values convergence analysis of GMRES with approximations to both \(A_{\gamma}\) and \(S\), similar results can be obtained by following the techniques in \cite{klawonn1999block, benzi2011field}. However, this extension is beyond the scope of this work.

In the following section, we will validate our theoretical findings by considering two benchmark problems: the driven cavity flow and flow past a backward-facing step. To validate and assess the actual dependence of convergence rates on the Reynolds number, we use exact solves for the (1,1) and (2,2) blocks of the preconditioner.

	\section{Numerical experiments}\label{sec:num}
	In this section we present results of numerical experiments with $Q_1-Q_1$ and $Q_2-Q_2$ elements.  For pressure  stabilization, we use the one from \cite{dohrmann2004stabilized}:
	\begin{equation*}
		s(p_h,q_h)=	(p_h-\pi_{h,T}p_h, q_h-\pi_{h,T}q_h),
	\end{equation*}
	where $\pi_{h,T}$ is the elementwise $L_2$-projection onto polynomials of degree $k-1$.
	
	
	We consider two 2D \rev{Navier--Stokes} flows to access the method accuracy and solver performance.\smallskip
	
	{\bf Problem 1 (driven cavity):} The tight cavity flow problem is defined  by  $\Omega=[-1,1]^2$, $\bm{f}=0$, $u_1(x,1)=1$, $u_2(x,1)=0$ for $x\in(-1,1)$, and $\bm{u}=0$ on the remainder of the boundary. The Reynolds number is given by $Re=2/\nu$.\smallskip
	
	{\bf Problem 2 (backward facing step):} It represents a slow flow in a rectangular duct with a sudden expansion, or flow over a step. The L-shaped domain is generated by taking the complement in $\Omega = (-1, L) \times (-1, 1)$ of the square $(-1, 0] \times (-1, 0]$. A Poiseuille flow profile $u_1(x,y) = 4y(1-y)$ is imposed on the inflow boundary ($x = -1; 0 \leq y \leq 1$), and a no-flow (zero-velocity) condition is imposed on the top and bottom walls. A Neumann condition of zero normal stress is applied at the outflow boundary.
	
	For the backward-facing step problem, the average inflow is defined as
	\begin{equation}
		U = \int_{0}^{1} \frac{1}{\mathcal{H} - h_s} u_1(-1, y) , dy,
	\end{equation}
	where $h_s = 1$ is the step size, and $\bm{u} = (u_1, u_2)$. The Reynolds number is defined as $Re = U \mathcal{H} / \nu$, where $\mathcal{H}$ is the height of the channel and $U$ is the mean inflow velocity. It follows that $Re = 4 / (3\nu)$.
	\smallskip
	
	Our numerical setup is as follows:
	\begin{itemize}
		\item The nonlinear algebraic system for the steady-state Navier-Stokes equations is solved using Picard iterations starting from zero velocity.
		\item In each Picard iteration, the discrete Oseen problem is solved using the preconditioned GMRES method.
		\item The stopping criterion for the inner linear solver is GMRES achieving a maximum of 400 iterations or a residual of $ \| r_k\| \leq 10^{-5}\| r_0\|$. For the outer nonlinear solver, the criterion is a maximum of 100 iterations or a nonlinear residual of $ \| s_k\| \leq 10^{-5}\| s_0\|$.
		\item The backward-facing step problem is solved for $Re=150$ with $L=5$ and $Re=800$ with $L=24$. The Reynolds number $Re=800$ is close to the first bifurcation point where the flow becomes unsteady~\cite{gresho1993steady}.
		\item For the driven cavity problem, we consider $Re=1000$, $Re=3200$, and $Re=5000$.
	\end{itemize}
	A zero initial guess was used for GMRES in all experiments.
	
	The choice of the grad-div stabilization parameter $\gamma$ is delicate. Analysis of the preconditioned GMRES suggests that a large $\gamma$ would benefit faster convergence. However, an excessively large parameter may ``overstabilize'' the problem, yielding less accurate finite element solutions to the Navier-Stokes system~\cite{olshanskii2009grad}. While there is substantial literature on the optimal choice of $\gamma$ for inf-sup stable elements (see, e.g., \cite{gelhard2005stabilized, olshanskii2009grad, jenkins2014parameter, de2016grad}), the question is less studied for equal-order elements. Here, we adopt $\gamma=0.1$, a common choice in the literature on grad-div stabilization and AL preconditioners for inf-sup stable elements. Besides iteration numbers, we will monitor the accuracy of FE solutions with this choice of $\gamma$.

	In addition to the preconditioner $\mathcal{\widehat{P}}$ defined in \eqref{eq:prec-inexact}, we also run the same experiments with the pressure convection-diffusion preconditioner (PCD) from \cite{elman2014finite} for comparison. Moreover, we propose an extension of the PCD preconditioner for grad-div stabilized problems. Specifically, we consider:
	\begin{equation} 
		\mathcal{P}_p =\begin{pmatrix}
			A_{\gamma}  & B^T\\
			0 &  -M_s
		\end{pmatrix},
	\end{equation}
	with 
	\begin{equation}
		M_s = A_pF_p^{-1}Q,
	\end{equation}
	where $A_p$ is the discrete pressure  Laplacian, $F_p=(\nu+\gamma) A_p  +N_p$ where $N_p$ is the ``pressure  advection matrix'' (cf.~\cite{elman2014finite}), and $Q$ is the vector mass matrix. 
	
	All numerical results reported below are produced using IFISS \cite{ers07}.

	\subsection{Driven cavity problem}
	
	\begin{figure}[h]
		\centering
		\includegraphics[width=0.49\linewidth]{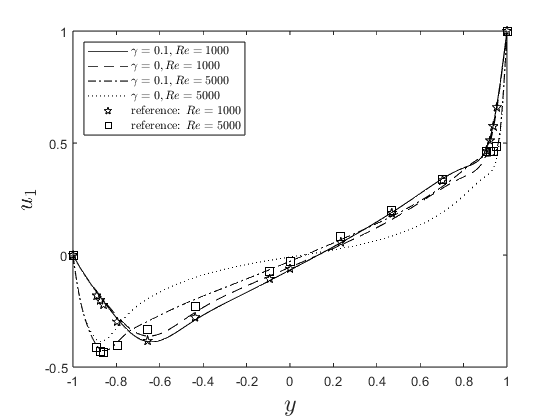} 
		\includegraphics[width=0.49\linewidth]{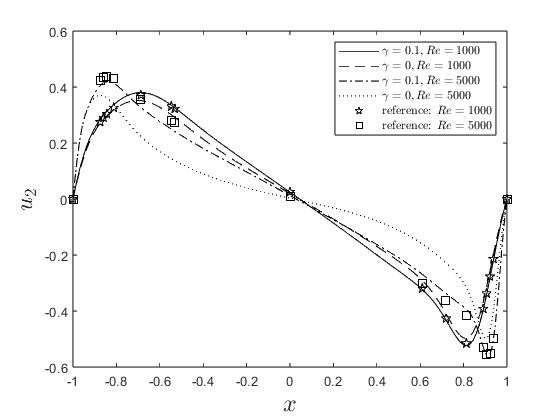}  
		\caption{$Q_1-Q_1$ FE solutions with and without grad-div stabilization for $Re=1000$ with $h=1/64$ and  $Re=5000$ with $h=1/128$. Left panel: $u_1$-component of velocity along the vertical center line of the cavity; Right panel: $u_2$-component of velocity along the horizontal center line of the cavity. The reference data is from \cite{ghia1982high}. }\label{fig-Q1-ux-uy} 
	\end{figure}
	\begin{figure}[h] 
		\centering
		\includegraphics[width=0.49\linewidth]{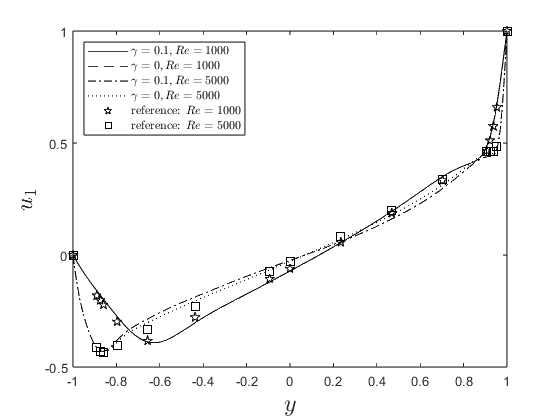} 
		\includegraphics[width=0.49\linewidth]{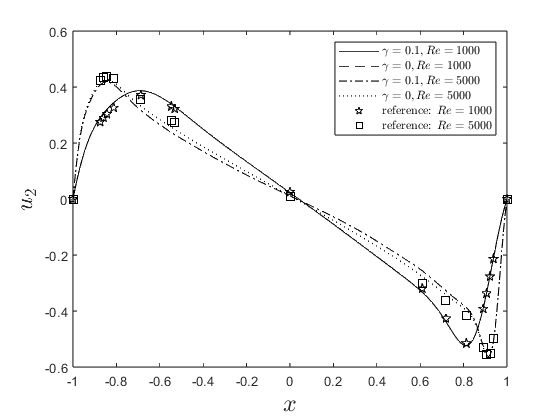}  
		\caption{$Q_2-Q_2$  FE solutions with and without grad-div stabilization for $Re=1000$ and $Re=5000$ with $h=1/64$. Left panel: $u_1$-component of velocity along the vertical center line of the cavity; Right panel: $u_2$-component of velocity along the horizontal center line of the cavity. The reference data is from \cite{ghia1982high}. } \label{fig-Q2-ux-uy}
	\end{figure}
	
	We start with the driven cavity problem. To assess the accuracy of the grad-div stabilized equal-order FE methods, we compare the velocity profiles of the FE solutions with reference data from \cite{ghia1982high}. Figure \ref{fig-Q1-ux-uy} displays $Q_1$ velocity solutions along cavity cross-sections for two Reynolds numbers. It is evident that using $\gamma=0.1$ yields \emph{significantly improved} approximations compared to $\gamma=0$, particularly noticeable for $Re=5000$. Figure \ref{fig-Q2-ux-uy} demonstrates that $Q_2$ elements produce very similar results with both $\gamma=0.1$ and $\gamma=0$, both aligning well with the reference data.

%

	\begin{figure}[h]
		\centering
		\includegraphics[width=0.49\linewidth]{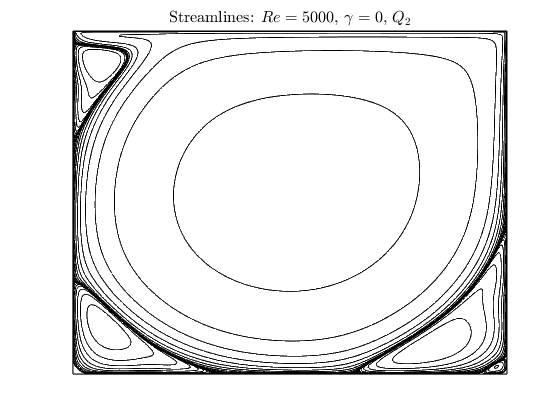} 
		\includegraphics[width=0.49\linewidth]{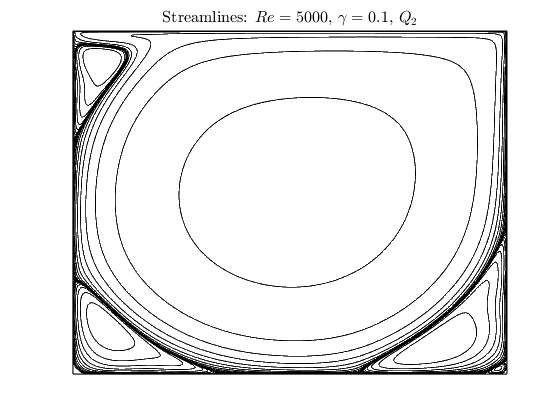}  
		\caption{Exponentially spaced streamlines of $Q_2-Q_2$  FE solutions without and with grad-div stabilization  with $h=1/64$ for   $Re=5000$.
		} \label{fig-stream1-5000}
	\end{figure}

	
\rev{ The computed vortex dynamics align with patterns extensively documented in the literature. For $Re = 1000$, the primary vortex and bottom secondary vortices are observed, while a vortex near the upstream upper corner appears at $Re = 3200$. At $Re = 5000$, tertiary vortices become visible. Figure~\ref{fig-stream1-5000} illustrates the flow at a Reynolds number of 5000, showing a very similar pattern recovered using $Q_2$  equal-order elements with $\gamma=0$ and $\gamma=0.1$.
}

	Adding the grad-div term with $\gamma=0.1$ to the equations does not compromise the accuracy of the finite element solutions; in fact, it notably enhances accuracy, particularly when lower-order elements are used to simulate higher Reynolds number flows. Therefore, we are now interested in examining the impact of this augmentation and proposed preconditioning on the efficiency of the algebraic solver.


	\begin{table}[h]
		\caption{Driven cavity: The Picard iteration counts (and the average of number of preconditioned GMRES iterations) for $Re=1000$. } 
		\centering
		\begin{tabular}{c|cccc |cccc} 
			\toprule
			Preconditioner & \multicolumn{4}{c}{$\mathcal{\widehat{P}}$}  & \multicolumn{4}{c}{$\mathcal{P}_p$}\\
			\hline
			elements   & \multicolumn{2}{c}{$Q_1$}  & \multicolumn{2}{c|}{$Q_2$}
			& \multicolumn{2}{c}{$Q_1$}  & \multicolumn{2}{c}{  $Q_2$} \\	
			\cmidrule{2-3} \cmidrule{4-5}  \cmidrule{6-7} \cmidrule{8-9}
			$h$	&$\gamma =0.1$   &$\gamma= 0$   &  $\gamma =0.1$ & $\gamma =0$    &$\gamma =0.1$   &$\gamma= 0$   &  $\gamma =0.1$ & $\gamma =0$     \\ 
			\midrule
			1/8     &8(15)  &- \cmmnt{100(61)} &8(16)  &15(289)  &8(15) &- \cmmnt{100(97)}   &8(16)  &15(98) \\ 
			\hline
			1/16    &8(15)  &11(128)           &6(15)  &7(400)   &8(15)   &11(79)           &6(16)   &7(79)\\ 
			\hline
			1/32    &6(15)  &7(230)            &5(14)  &5(388)   &6(14)   &7(81)          &5(15)   &5(59) \\ 
			\hline
			1/64    &5(14)  &5(331)            &4(14)  &4(384)   &5(14)   &5(60)          &4(15)   &4(50)\\ 
			\hline
		\end{tabular}\label{tb:cavity-iteration-1000}
	\end{table}
	
	\begin{table}[h]
		\caption{Driven cavity: The Picard iteration counts (and the average of number of preconditioned GMRES iteration)  for $Re=3200$. }  
		\centering
		\begin{tabular}{c|cccc |cccc} 
			\toprule
			Preconditioner & \multicolumn{4}{c}{$\mathcal{\widehat{P}}$}  & \multicolumn{4}{c}{$\mathcal{P}_p$}\\
			\hline
			elements   & \multicolumn{2}{c}{$Q_1$}  & \multicolumn{2}{c|}{$Q_2$}
			& \multicolumn{2}{c}{$Q_1$}  & \multicolumn{2}{c}{  $Q_2$} \\	
			\cmidrule{2-3} \cmidrule{4-5}  \cmidrule{6-7} \cmidrule{8-9}
			$h$	&$\gamma =0.1$   &$\gamma= 0$   &  $\gamma =0.1$ & $\gamma =0$    &$\gamma =0.1$   &$\gamma= 0$   &  $\gamma =0.1$ & $\gamma =0$     \\ 
			\midrule
			1/8     &10(16)   &- \cmmnt{100(57)}          &9(19)   &-  \cmmnt{100(400)}   &10(17)   &- \cmmnt{100(253)}    &9(20)   &- \cmmnt{100(400)}  \\ 
			\hline
			1/16    &9(16)   &- \cmmnt{100(128)}         &8(16)   &- \cmmnt{100(400)}   &9(18)    &- \cmmnt{100(80)}     &8(18)   &- \cmmnt{100(250)}  \\ 
			\hline
			1/32    &8(16)   &- \cmmnt{100(315)}         &7(16)   &16(400)              &8(17)   &- \cmmnt{100(156)}     &7(17)   &9(177)  \\ 
			\hline
			1/64    &7(16)   &11(400)                    &5(14)   &10(400)              &7(16)   &9(167)                 &5(16)   &6(125)  \\ 
			\hline
		\end{tabular}\label{tb:cavity-iteration-3200}
	\end{table}
	
	\begin{table}[h]
		\caption{Driven cavity: The Picard iteration counts (and the average of number of preconditioned GMRES iterations)   for $Re=5000$. } 
		\centering
		\begin{tabular}{c|cccc |cccc} 
			\toprule
			Preconditioner & \multicolumn{4}{c}{$\mathcal{\widehat{P}}$}  & \multicolumn{4}{c}{$\mathcal{P}_p$}\\
			\hline
			elements   & \multicolumn{2}{c}{$Q_1$}  & \multicolumn{2}{c|}{$Q_2$}
			& \multicolumn{2}{c}{$Q_1$}  & \multicolumn{2}{c}{  $Q_2$} \\	
			\cmidrule{2-3} \cmidrule{4-5}  \cmidrule{6-7} \cmidrule{8-9}
			$h$	&$\gamma =0.1$   &$\gamma= 0$   &  $\gamma =0.1$ & $\gamma =0$    &$\gamma =0.1$   &$\gamma= 0$   &  $\gamma =0.1$ & $\gamma =0$     \\ 
			\midrule
			1/8     &12(16)   &- \cmmnt{100(59)}     &11(21)   &-  \cmmnt{100(370)}           &12(18)   &- \cmmnt{100(289)}   &11(22)   &- \cmmnt{100(400)}  \\ 
			\hline
			1/16    &10(17)   &- \cmmnt{100(159)}    &9(18)   &- \cmmnt{100(400)}         &10(18)   &- \cmmnt{100(400)}    &10(20)   &- \cmmnt{100(400)}  \\ 
			\hline
			1/32    &9(17)   &- \cmmnt{100(308)}    &9(16)   &- \cmmnt{100(400)}         &9(18)   &- \cmmnt{100(221)}    &9(17)  &- \cmmnt{100(285)}  \\ 
			\hline
			1/64    &7(15)   &- \cmmnt{100(400)}    &7(15)   &11(400)                     &7(16)   &- \cmmnt{100(221)}    &7(18)   &7(163)  \\ 
			\hline
		\end{tabular}\label{tb:cavity-iteration-5000}
	\end{table}
	
	Table \ref{tb:cavity-iteration-1000} reports iteration numbers for $Re=1000$ with varying mesh sizes. The notation ``-'' in this and subsequent tables indicates that Picard's method reaches the maximum of 100 iterations without convergence. The results indicate that using $\gamma=0.1$ leads to convergence in only a few iterations for both Picard and GMRES methods, whereas for $\gamma=0$, GMRES requires significantly more iterations to satisfy the convergence criterion. Additionally, we observe that $Q_2$ elements require slightly more iterations than $Q_1$ elements to converge. Furthermore, our preconditioner ($\mathcal{\widehat{P}}$) with grad-div stabilization proves to be robust with respect to mesh size and elements. For $\gamma=0.1$, the modified PCD demonstrates similar results to $\mathcal{\widehat{P}}$. However, it is worth noting that the modified PCD is somewhat more computationally expensive and requires the definition of pressure ``convection-diffusion'' matrices not present in the original formulation, limiting its applicability.
	
	Tables~\ref{tb:cavity-iteration-3200} and \ref{tb:cavity-iteration-5000} present results for increasing Reynolds numbers. We see that the simple  preconditioner $\mathcal{\widehat{P}}$ with grad-div stabilization ($\gamma=0.1$) remains robust with respect to both $Re$ and mesh size. For $\gamma=0$, the GMRES with $\mathcal{\widehat{P}}$ and $\mathcal{P}_p$ preconditioners  do not converge, except for the PCD preconditioner for $Q_2$ elements on finer meshes.

	\subsection{Backward facing step problem} 
	Following the same line of analysis as above, we begin by comparing the accuracy of the computed solutions for augmented (i.e., grad-div stabilized with $\gamma=0.1$) and non-augmented problems ($\gamma=0$). The key statistics of interest for the backward-facing step problem are the streamlines' separation and reattachment points on the lower and upper walls.
	
	\begin{table}[h]
		\caption{Backward facing step: Numerical results for $h=1/32$ for $Q_1$ and $Q_2$ elements: $r_1$ is the reattachment point for the bottom vortex, $r_2$ is the left separation point for the upper vortex, and $r_3$ is the right reattachment point for the upper vortex. }
		\centering \small
		\begin{tabular}{l l l l l} 
			Configuration	&method   &$r_1$   & $r_2$    & $r_3$    \\ 
			\hline
			$Re=150$                &  &  & &\\
			$\gamma=0.1$ &$Q_1$  &4.22   &- &- \\
			$\gamma=0$ &$Q_1$   &4.00  &- &-\\
			$\gamma=0.1$ &$Q_2$  &4.25   &- &- \\
			$\gamma=0$ &$Q_2$   &4.25   &- &-\\
			Ref.\cite{olshanskii2002low}  &FE     &4.00  &- &- \\
			Ref.\cite{morgan1984notes}        &Experimental   &4.50  &- &-\\
			\hline
			$Re=800$  &  &  & &\\
			$\gamma=0.1$ &$Q_1$  &11.50  &9.41  &20.22  \\
			$\gamma=0$ &$Q_1$  &12.09  &9.88 &20.22 \\
			$\gamma=0.1$ &$Q_2$  &11.00   &8.72  &20.20   \\
			$\gamma=0$ &$Q_2$  &11.02  &8.73  &20.20  \\
			Ref.\cite{olshanskii2002low}  &FE    &11.88  &9.70  &20.42\\
			Ref.\cite{gartling1990test}          &FD         &12.20  &9.70  &20.96 \\
			Ref.\cite{canuto1998stabilized}        &Spectral   &11.94  &9.78  &20.92 \\
			Ref.\cite{armaly1983experimental}      &Experimental   &14.20  &- &-\\
			\hline
		\end{tabular}\label{tb:reatachment}
	\end{table} 
	
	Table \ref{tb:reatachment} shows these statistics for both Reynolds numbers and provides results from the literature for comparison.
	For $Re=150$, a vortex near the upper wall is not forming, so only one reattachment point is reported.
	
	From Table~\ref{tb:reatachment} we see that the accuracy of stabilized and non-stabilized solutions is similar and the results compare reasonably well to those found in the literature.

	\begin{figure}[h]  
		\centering
		\includegraphics[width=0.49\linewidth]{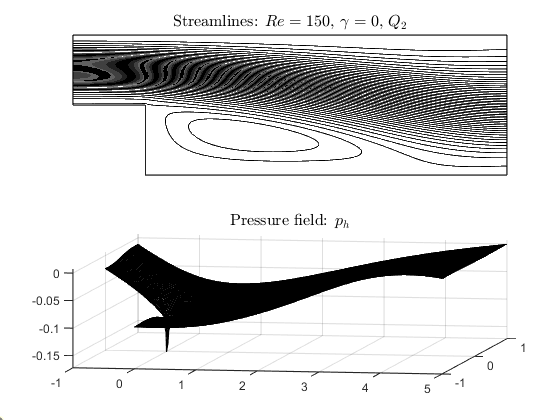}
		\includegraphics[width=0.49\linewidth]{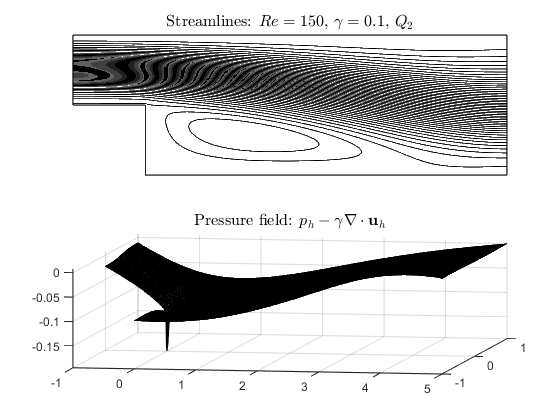}
		\caption{Flow over the backward-facing step for $Re=150$. Left panel:  $Q_2$ solution with $\gamma=0$, $h=1/32$.  Right panel:  $Q_2$ solution with $\gamma=0.1$, $h=1/32$.}\label{fig-Q2-150-step}
	\end{figure}

	\begin{figure}[h] 
		\centering
		\includegraphics[width=0.49\linewidth]{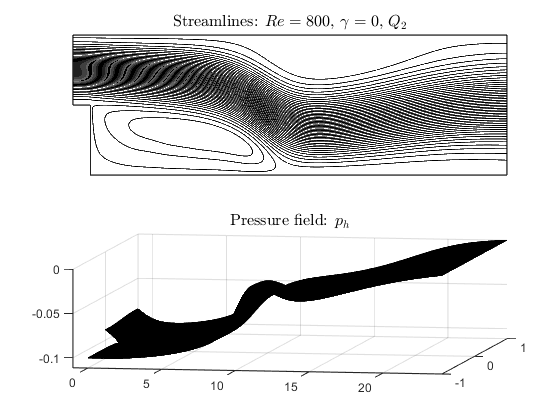}
		\includegraphics[width=0.49\linewidth]{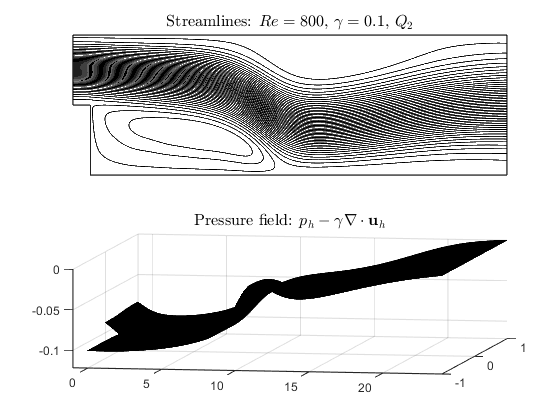}
		\caption{Flow over the backward-facing step for $Re=800$. Left panel:  $Q_2$ solution with $\gamma=0$, $h=1/32$.  Right panel:  $Q_2$ solution with $\gamma=0.1$, $h=1/32$.}\label{fig-Q1Q2-800-step} 
	\end{figure}
	
	Figures \ref{fig-Q2-150-step} and \ref{fig-Q1Q2-800-step} show streamlines and pressure of $Q_2$ solutions for $Re=150$ and $Re=800$. This time, solutions  with and without grad-div stabilization are visually indistinguishable.     
	The computed  $Q_1$ solutions with $h=1/32$ and both $\gamma=0.1$ and $\gamma=0$ were visually similar to those shown in the figures.

	\begin{table}[h]
		\caption{Backward facing step: The Picard iteration counts (and the average of number of preconditioned GMRES iteration) for $Re=150$. }
		\centering  
		\begin{tabular}{c|cccc |cccc} 
			\toprule
			Preconditioner & \multicolumn{4}{c}{$\mathcal{\widehat{P}}$}  & \multicolumn{4}{c}{$\mathcal{P}_p$}\\
			\hline
			elements   & \multicolumn{2}{c}{$Q_1$}  & \multicolumn{2}{c|}{$Q_2$}
			& \multicolumn{2}{c}{$Q_1$}  & \multicolumn{2}{c}{  $Q_2$} \\	
			\cmidrule{2-3} \cmidrule{4-5}  \cmidrule{6-7} \cmidrule{8-9}
			$h$	&$\gamma =0.1$   &$\gamma= 0$   &  $\gamma =0.1$ & $\gamma =0$    &$\gamma =0.1$   &$\gamma= 0$   &  $\gamma =0.1$ & $\gamma =0$   \\ 
			\midrule
			1/8     &12(27)   &11(118)    &11(28)   &11(255)    &12(22)    &11(54)   &11(26)   &11(49)   \\ 
			\hline
			1/16    &11(27)   &10(185)    &9(28)   &9(278)      &11(26)    &10(50)   &9(28)    &9(61) \\
			\hline
			1/32    &9(28)    &9(240)     &8(27)   &8(289)      &9(27)    &9(61)     &8(28)   &8(85) \\
			\hline
			1/64    &8(27)    &8(268)     &6(27)   &6(289)      &8(26)    &8(80)     &6(27)     &6(289) \\ 
			\hline
		\end{tabular}\label{tb:it-150-step} 
	\end{table}
	
	\begin{table}[h]
		\caption{Backward facing step: The Picard iteration counts (and the average of number of preconditioned GMRES iteration)  for $Re=800$.}
		\centering \hskip-1ex
		\begin{tabular}{c|cccc |cccc} 
			\toprule
			Preconditioner & \multicolumn{4}{c}{$\mathcal{\widehat{P}}$}  & \multicolumn{4}{c}{$\mathcal{P}_p$}\\
			\hline
			elements   & \multicolumn{2}{c}{$Q_1$}  & \multicolumn{2}{c|}{$Q_2$}
			& \multicolumn{2}{c}{$Q_1$}  & \multicolumn{2}{c}{  $Q_2$} \\	
			\cmidrule{2-3} \cmidrule{4-5}  \cmidrule{6-7} \cmidrule{8-9}
			$h$	 &$\gamma =0.1$ &$\gamma= 0$  & $\gamma =0.1$    & $\gamma =0$   &$\gamma =0.1$   &$\gamma= 0$   & $\gamma =0.1$    & $\gamma =0$  \\ 
			\midrule
			1/8    &80(45)       &-                 &62(46)   &97(400)              &80(27)           &-              &62(34)    &61(184)\\  
			\hline
			1/16  &64(46)       &53(400)            &48(46)   &-    &64(34)          &57(218)         &48(37)    &48(131)\\ 
			\hline
			1/32  &49(46)       &72(400)            &35(46)   &-                   
			&49(38)            &47(147)         &35(38)    &35(130)\\ 
			\hline
			1/64  &37(46)       &81(400)          &26(48)   &-          &37(37)            &36(136)         &26(37)    &25(186)  \\  
			\hline
		\end{tabular}\label{tb:it-800-step} 
	\end{table}
	
	Next, we present the iteration counts for Picard's method and the preconditioned GMRES methods using two preconditioners. Table \ref{tb:it-150-step} reveals that for $Re=150$, the nonlinear iteration numbers are nearly independent of $\gamma$. However, grad-div stabilization significantly impacts the inner GMRES iterations: incorporating grad-div stabilization notably reduces the GMRES iteration numbers, particularly for $Q_2$ elements. Moreover, the number of GMRES iterations with grad-div stabilization remains essentially unaffected by mesh size and the choice of elements, consistent with our theoretical findings. A similar observation applies to the (modified) PCD preconditioner, albeit requiring slightly fewer iterations for inner GMRES but with slightly higher computational cost.
	
	In Table \ref{tb:it-800-step}, we present the iteration numbers for $Re=800$. We notice an increase in Picard's iteration numbers compared to the $Re=150$ case, as expected, \rev{as well as the slight increase in GMRES iterations. At the same time,} the reduction in GMRES iteration counts with $\gamma=0.1$ is even more pronounced in this scenario than for $Re=150$. Once again, the iteration count of GMRES with preconditioner $\mathcal{\widehat{P}}$ and $\gamma=0.1$ remains unaffected by variations in mesh parameters and elements, consistent with our convergence analysis.\rev{The number of GMRES iterations per nonlinear step fluctuates very little from the averaged numbers reported in the tables.}

	\section{Conclusions}\label{sec:con}
	In this paper, we investigated a grad-div stabilized equal-order finite element discretization of the Oseen problem, focusing primarily on its algebraic properties and preconditioning. However, we also considered the accuracy aspect of the augmentation. We proposed a block triangular preconditioner akin to the augmented Lagrangian (AL) approach and applied field-of-values analysis to demonstrate that the convergence rate of GMRES with this ideal preconditioner is  mesh-independent. Numerical experiments supported our theoretical findings, indicating that a stabilization parameter of around $0.1$ can be used consistently across all Reynolds numbers and independent of mesh size.
	
	Furthermore, numerical experiments showed that the augmented solver outperforms the block-triangular preconditioner with the state-of-the-art pressure convection-diffusion Schur complement approximation. We also proposed a suitable modification of the latter for the augmented case. Both the AL-type preconditioner and the modified PCD demonstrated convergence rates insensitive to variations in Reynolds number and discretization parameters.
	
	The approximation to the $(1,1)$ block in the augmented system was not studied in this work.  We expect that an LU factorization or keeping a block upper-triangle part of this block should be an efficient strategy. We plan to explore such possibility in the future.
	


\bibliographystyle{siam}
\bibliography{PkPkBib}

\end{document}